\documentclass[11pt, letterpaper]{amsart}
\usepackage[left=1in,right=1in,bottom=0.5in,top=0.8in]{geometry}
\usepackage{amsfonts}
\usepackage{amsmath, amssymb}
\usepackage{graphicx}
\usepackage[font=small,labelfont=bf]{caption}
\usepackage{epstopdf}
\usepackage[pdfpagelabels,hyperindex]{hyperref}
\usepackage{xcolor}
\usepackage{amsthm}
\usepackage{float}
\usepackage{pgfplots}
\usepackage{listings}
\usepackage{longtable}
\usepackage{mathrsfs}

\newtheorem{thm}{Theorem}[section]
\newtheorem{cor}[thm]{Corollary}
\newtheorem{prop}[thm]{Proposition}

\theoremstyle{definition}
\newtheorem{defn}[thm]{Definition}

\theoremstyle{remark}

\pgfplotsset{compat=1.9}

\newcommand{\RN}[1]{%
	\textup{\uppercase\expandafter{\romannumeral#1}}%
}
\newcounter{x}

\numberwithin{equation}{section}

\author[Chutian Ma]{Chutian Ma}
\email{cma27@jhu.edu}

\title{On the Fourier Truncation Method for the Rough Data Cubic Defocusing NLW on  $\mathbb{H}^3$}
\begin{document}
	\begin{abstract}
		In this paper, we study the cubic defocusing nonlinear wave equation on the three dimensional hyperbolic space. We use the Fourier truncation method to show that the equation is globally well-posed and scatters if the initial data lies in $H^s(\mathbb{H}^3)$, $s>\frac{182}{201}\approx 0.905$.
		
	\end{abstract}
	\maketitle
	
	\section{Introduction}
	
	In this paper we study the following cubic defocusing nonlinear wave equation on the hyperbolic space $\mathbb{H}^3$ 	\begin{equation}\label{eq:01}
		\left\{\begin{aligned}
			&u_{tt}-\Delta_{\mathbb{H}^3}u+F(u)=0\\
			&u(0,x)=u_0(x),u_t(0,x)=u_1(x)
		\end{aligned}
		\right.
	\end{equation}
	with $F(u)=u^3$. The initial data $(u_0,u_1)$ is assumed to lie in the Sobolev space $H^s(\mathbb{H}^3)$, whose norm is defined as \\
	\begin{equation*}
		\|f\|_{H^s(\mathbb{H}^3)}:=\|(-\Delta_{\mathbb{H}^3})^\frac{s}{2}f\|_{L^p(\mathbb{H}^3)}
	\end{equation*}
	On Euclidean space, it is well known\cite{lindblad1995existence} that \ref{eq:01} is locally well-posed if 
	$s\geq\frac{1}{2}$. Local and global well-posedness are defined in the usual way:
	\begin{defn}[Local Well-posedness]
		The initial value problem (\ref{eq:01}) is said to be locally well-posed if there exists an interval $I\subset \mathbb{R}$ containing time 0 such that:
		\begin{enumerate}
			\item There exists a unique solution $u\in C_tH^\frac{1}{2}_x(I)\cap L^4_{t,x}(I)$ and $u_t\in C_tH^{-\frac{1}{2}}_x(I)$
			\item The solution depends continuously on the initial data in the sense of the topology of the function space in (1)
		\end{enumerate}
	\end{defn}
	\begin{defn}[Global Well-posedness]
		The initial value problem \ref{eq:01} is said to be globally well-posed if for any interval $I\subset \mathbb{R}$,
		\begin{enumerate}
			\item There exists a unique solution $u\in C_tH^\frac{1}{2}_x(I)\cap L^4_{t,x}(I)$ and $u_t\in C_tH^{-\frac{1}{2}}_x(I)$
			\item The solution depends continuously on the initial data in the sense of the topology of the function space in (1)
		\end{enumerate}
	\end{defn}	
	$H^1$ solution has conserved energy
	\begin{equation}\label{eq:02}
		E(u)(t)=\int \frac{1}{2}|\nabla u|^2+\frac{1}{2}u_t^2+\frac{1}{4}u^4
	\end{equation}
	Global well-posedness of $H^1$ solution to (\ref{eq:01}) follows from (\ref{eq:02}) and local well-posedness. In fact, it is conjectured that global well-posedness and scattering are true for all the initial data in $H^s\times H^{s-1}$ with $s\geq\frac{1}{2}$. 
	\begin{defn}
	    The solution to (\ref{eq:01}) is said to scatter to a linear solution in $H^s\times H^{s-1}$ if there exist $(u_0^+,u_1^+),(u_0^-,u_1^-)\in H^s\times H^{s-1}$ such that
	    \begin{equation*}
	      \begin{aligned}
	    	&lim_{t\rightarrow +\infty}\|(u,u_t)-S(t)(u_0^+,u_1^+)\|_{H^s\times H^{s-1}}=0\\
	    	&im_{t\rightarrow -\infty}\|(u,u_t)-S(t)(u_0^-,u_1^-)\|_{H^s\times H^{s-1}}=0
	      \end{aligned}
	    \end{equation*}	
	    where $S(t)(u_0,u_1)$ is the propagator of the free wave equation.
	\end{defn}

	The nonlinear wave equations on the Euclidean space have a rich collection of rough data results.
	
	\cite{kenig2000global} proved global well-posedness for (\ref{eq:01}) in $\mathbb{R}^3$ for $\frac{3}{4}<s<1$. In particular, they showed that the $H^s\times H^{s-1}$ norm of the solution grows at most like polynomial in time. The minimum regularity required for global solutions have been improved, see \cite{gallagher2003global}, \cite{bahouri2006global}, \cite{roy2007global} for example. Recently, \cite{dodson2019global} proved global well-posedness for radial initial data in $H^s\times H^{s-1}$ when $s>\frac{1}{2}$ and in a later paper\cite{dodson2018global} extend the result to the critical Sobolev space $H^\frac{1}{2}\times H^{-\frac{1}{2}}$(with the radial assumption). See also \cite{dodson2022global}\\

	In the hyperbolic space, due to geometric reasons, the free wave equation has stronger dispersive estimates compared to the Euclidean case. Pierfelice in \cite{pierfelice2008weighted} obtained weighted Strichartz estimates for \ref{eq:01} with radial data on Damek-Ricci spaces, which include hyperbolic spaces. See also \cite{ionescu2000fourier}. For 3-dimensional hyperbolic space specifically, Metcalfe and Taylor\cite{metcalfe2011nonlinear} investigated (\ref{eq:01}) on $\mathbb{H}^3$ and obtained Strichartz type inequalities. The admissible pairs in their result are broader than those in their counterpart in Euclidean space. See (\ref{thmStric}). As an application, they proved small data global results for (\ref{eq:01}) with initial data in $H^\frac{1}{2}\times H^{-\frac{1}{2}}$. Anker and Pierfelice in \cite{anker2014wave} also investigated the Klein-Gordon equation on hyperbolic spaces $\mathbb{H}^d,d\geq 2$, of which the wave equation is a special case. They proved Strichartz estimates for a large family of admissible pairs.\\

    The Laplacian $-\Delta$ on $\mathbb{H}^d$ is positive with spectrum $[\rho,+\infty)$, where $\rho=\frac{d-1}{2}$. A variant of (\ref{eq:01}) is also considered, which is called the shifted wave equation(\ref{eq:01shift}).
    \begin{equation}\label{eq:01shift}
		\left\{\begin{aligned}
			&u_{tt}-\Delta_{\mathbb{H}^3}-\rho^2u+F(u)=0\\
			&u(0,x)=u_0(x),u_t(0,x)=u_1(x)
		\end{aligned}
		\right.
	\end{equation}
    Fontaine in \cite{jean1997semilinear} studied the shifted wave equation on 2d and 3d. Tataru in \cite{tataru2001strichartz} obtained dispersive estimates for shifted wave equations in hyperbolic space. Anker, Pierfelice and Vallarino obtained Strichartz estimates for (\ref{eq:01shift}) for a large family of admissible pairs. As an application, Shen and Staffilani in \cite{shen2016semi} proved global well-posedness and scattering for (\ref{eq:01})(and more general power) with $H^{\frac{1}{2},\frac{1}{2}}\times H^{-\frac{1}{2},\frac{1}{2}}$ initial data. 
    
    This paper aims to study the initial value problem (\ref{eq:01}) with low regularity initial data in $H^s\times H^{s-1}$, $s<1$.
	Our main result is the following theorem:
	
	\begin{thm}
	  The initial value problem (\ref{eq:01}) is globally well-posed and scatters for initial data in $H^s\times H^{s-1}$, assuming that $s>\frac{182}{201}\approx 0.905$. 
	
	\end{thm}

	Our proof uses the Fourier truncation method, using which the author also proved similar results for nonlinear cubic Schr\'{o}dinger equation in \cite{ma2022scattering}. The major obstacle to reach global results from local ones is the lack of a conserved quantity that controls the $H^\frac{1}{2}\times H^{-\frac{1}{2}}$ norm of the solution. To counter that, we split the initial data into two parts, one of high frequency and another of low frequency. We evolve the high frequency data by the linear equation. The low frequency data is evolved by the standard cubic equation. This results in a correction term which starts with zero initial data and satisfies the cubic equation plus some error terms in its nonlinearity. We first evolve the three pieces in a small time interval. The low mode has conserved energy. What we need to estimate is the energy of the correction term. We will add the correction term to the low mode when we start the evolution on the next time interval. So that the new low mode will start with data from the original low mode and the correction term, whereas the new correction term starts with zero initial data again. We divide time into subintervals and repeat this procedures step by step. Finally, we show that the energy increment the low mode gained from the correction terms is bounded.\\
	The structure of the paper is organized as following. In section 2, we present some preliminary results we need for our proof. In section 3, we elaborate our scheme using the fourier truncation method and estimate the energy of the low mode on a given time interval. In section 4, we combine our results in section 3 along with a Morawetz inequality in a bootstrap argument to show that the scattering norm stays bounded. Thus the global well-posedness and scattering follow.

	\pagebreak

	\section{Preliminary Results}
	In this section we present some of the preliminary results we will need later.
	\subsection{Analysis on $\mathbb{H}^3$}
	We shall define the hyperbolic space through the following model. We consider the standard Minkowski space $\mathbb{R}^{3+1}$ endowed with the metric $-dx_0^2+dx_1^2+dx_2^2+dx_3^2$ and the bilinear form $[x,y]=x_0y_0-x_1y_1-x_2y_2-x_3y_3$. Hyperbolic space $\mathbb{H}^3$ is defined as the submanifold satisfying $[x,x]=1$ whose metric is induced from the Minkowski space. \\
	We shall also use the following polar coordinates on the hyperbolic space. We use the pair $(r,\omega)\in \mathbb{R}\times \mathbb{R}^2$ to represent the point $(coshr,sinhr\omega)$ in the previous model of $\mathbb{H}^3$. r represents the distance from the point $(coshr,sinhr\omega)$ to the origin (1,0,0,0). \\
	The metric in the polar coordinates takes the form
	\begin{equation*}
	    g_{\mathbb{H}^3}=dr^2+sinh^2rd\omega^2	
	\end{equation*}
	And integrals shall be computed by
	\begin{equation*}
	    \int_{\mathbb{H}^3}	fd\mu=\int_0^\infty\int_{\mathbb{S}^2}f(r,\omega)sinh^2rdrd\omega
	\end{equation*}

    Similar to the Euclidean case, fourier transform can be defined for suitable functions. See \cite{helgason2001differential} for example. In fact, fourier transform on $\mathbb{H}^3$ takes functions on $\mathbb{H}^3$ to functions defined on $(\lambda,\theta)\in\mathbb{R}\times\mathbb{S}^2$ through the following transform and inversion formula:
    \begin{equation*}
    	\hat{f}(\lambda,\theta)=\int_{\mathbb{H}^3}f(x)[x,b(\omega)]^{i\lambda-1}d\mu(x)
    \end{equation*}
    \begin{equation*}
    	f(x)=\int_0^\infty\int_{\mathbb{S}^2}\tilde{f}(\lambda,\theta)[x,b(\omega)]^{-i\lambda-1}\frac{d\lambda d\theta}{|c(\lambda)|^2}
    \end{equation*}
    where $b(\omega)=[1,\omega]\in\mathbb{R}^4$ and 
    the Harish-Chandra function $c(\lambda)$ is defined by
    \begin{equation*}
    	c(\lambda)=C\frac{\Gamma(i\lambda)}{\Gamma(i\lambda+1)}
    \end{equation*}
    
    For $L^2$ integrable function, the following Plancherel formula holds:
    \begin{equation*}
    	\int_{\mathbb{H}^3}f_1(x)\overline{f_2(x)}d\mu=\int_0^\infty\int_{\mathbb{S}^2}\hat{f}_1(\lambda,\theta)\overline{\hat{f}_2(\lambda,\theta)}\frac{d\lambda d\theta}{|c(\lambda)|^2}
    \end{equation*}

    The Laplace-Beltrami operator is given in the polar form by
    \begin{equation*}
    	\Delta_{\mathbb{H}^3}=\partial_r^2+\frac{2coshr}{sinhr}\partial_r+\frac{1}{sinh^2r}\Delta_{\mathbb{S}^2}
    \end{equation*}

    $\Delta_{\mathbb{H}^3}$ is strictly positive definite. The spectrum of $\Delta_{\mathbb{H}^3}$ is $[1,\infty)$. In fact, the fourier transform of Laplace operator is given by
    \begin{equation}
        \widehat{(-\Delta_{\mathbb{H}^3}f)}=(\lambda^2+1)\hat{f}	
    \end{equation}
    Thus, we have the following Poincar\'e inequality
    \begin{prop}
    For $\alpha<\beta$, we have
    \begin{equation}
    	\|(-\Delta_{\mathbb{H}^3})^\frac{\alpha}{2}f\|\lesssim \|(-\Delta_{\mathbb{H}^3})^\frac{\beta}{2}f\|
    \end{equation}    	
    \end{prop}

	\subsection{Heat-Flow Based Frequency Projection}\hfill \\
	On Euclidean space, we define operators of the form $\chi(-\Delta)$, where $\chi$ is a compactly supported bump function to localize frequency. But unfortunately, operators defined in this way are not $L^p$ bounded on general manifolds. As such, we use the following heat-flow based operator as a substitution of the usual Littlewood Paley operators. The following definition and Bernstein inequalities are the same as in \cite{lawrie2018local}.

	\begin{equation} \label{LPO}
		\begin{aligned}
			&P_{\geq s}f=e^{s\Delta}f\\
			&P_{<s}f=f-P_{\geq s}f\\
			&P_sf=(-s\Delta)e^{s\Delta}f
		\end{aligned}
	\end{equation}
	
	\begin{prop}
	The operators defined in (\ref{LPO}) are bounded on $L^p(\mathbb{H}^3)$.
	\end{prop}
		
	Heuristically, $P_{\geq s}$ can be viewed as frequency localization to frequency $\leq s^{-\frac{1}{2}}$, whereas $P_s$ can be viewed as the frequency localization to frequency $\sim s^{-\frac{1}{2}}$, in the sense that they satisfy the following Bernstein inequalities:
	
	\begin{prop}[Bernstein Inequality]\label{BernIneq}
		\begin{equation*}
			\begin{aligned}
				&\|P_{<s}f\|_{L^p_x}\lesssim s^\frac{1}{2}\|\nabla f\|_{L^p_x}\\
				&\|\nabla P_{\geq s}f\|_{L^p_x} \lesssim s^{-\frac{1}{2}}\|f\|_{L^p_x}\\
				&\|\nabla P_{s}f\|_{L^p_x} \sim s^{-\frac{1}{2}}\|f\|_{L^p_x}
			\end{aligned}
		\end{equation*}
		
	\end{prop}

	\subsection{Strichartz Estimates}
	We will use the Strichartz estimates obtained in \cite{metcalfe2011nonlinear} for wave equation on the 3 dimensional hyperbolic space. Define $\mathcal{R}$ to be the set of $(p,q,\gamma)$ such that
	\begin{equation*}
		\begin{aligned}
		   &\frac{1}{p}+\frac{1}{q}\leq\frac{1}{2}\\
		   &p,q\geq 2\\
		   &\gamma=\frac{3}{2}-\frac{1}{p}-\frac{3}{q}
		\end{aligned}
	\end{equation*}	
	and $\mathcal{E}$ to be the set of $(p,q,\gamma)$
	\begin{equation*}
		\begin{aligned}
		    &\frac{1}{2}-\frac{1}{p}\leq\frac{1}{q}\leq\frac{1}{2}-\frac{1}{3p}\ and\ p>2\\
		    or\ &0<\frac{1}{q}<\frac{1}{3}\ and\ p=2\\
		    &\gamma=1-\frac{2}{q}
		\end{aligned}
	\end{equation*}
    We have the following Strichartz inequalities
    \begin{thm}[Strichartz estimates in \cite{metcalfe2011nonlinear}]\label{thmStric}
    Suppose $(p,q,\gamma)\in\mathcal{R}\cup \mathcal{E}$, then the mapping defined by
    \begin{equation}\label{eq:Stric01}
      \begin{aligned}
          &T:H^\gamma(\mathbb{H}^3)\rightarrow L^p_tL^q_x(\mathbb{R}\times \mathbb{H}^3)\\
          &Tf(t,x)=e^{it\sqrt{-\Delta}}f(x)
      \end{aligned}
    \end{equation}
    or equivalently
    \begin{equation}\label{eq:Stric02}
      \begin{aligned}
          &T^*:L^{p'}_tL^{q'}_x(\mathbb{R}\times \mathbb{H}^3)\rightarrow H^{-\gamma}(\mathbb{H}^3)\\
          &T^*F(x)=\int_{-\infty}^{+\infty}e^{-it\sqrt{-\Delta}}F(t,x)dt
      \end{aligned}
    \end{equation}
    is bounded.
    \end{thm}
    Note that the solution of \ref{eq:01} is given by
    \begin{equation}\label{eq:duh}
        u(t,x)=cos(t\sqrt{-\Delta})u_0+\frac{sin(t\sqrt{-\Delta})}{\sqrt{-\Delta}}u_1+\int_{s<t}\frac{sin((t-s)\sqrt{-\Delta})}{\sqrt{-\Delta}}f(s,x)ds
    \end{equation}
    Strichartz estimate for the solution follows from (\ref{thmStric}) and (\ref{eq:duh})
    \begin{cor}
    \hfill Let $I\subset\mathbb{R}$ and $t_0\in I$. Suppose $(p,q,\gamma)\in\mathcal{R}$, then we have the following estimate for the solution of \ref{eq:01}
    \begin{equation}
        \|u\|_{L^p_tL^q_x(I)}+\|u\|_{L^\infty_tH^\gamma_x(I)}+\|u_t\|_{L^\infty_tH^{\gamma-1}_x(I)}\lesssim_{p,q,\gamma}\|u_0\|_{H^\gamma_x}+\|u_1\|_{H^{\gamma-1}}+\|f\|_{L^{p'}_tL^{q'}_x(I)}	
    \end{equation}

    \end{cor}
	\begin{proof}
	    The estimates for $cos(t\sqrt{-\Delta})u_0$ and $\frac{sin(t\sqrt{-\Delta})}{\sqrt{-\Delta}}u_1$ follow immediately from (\ref{eq:Stric01}).
        Consider the operator W defined by
        \begin{equation}
            WF(t,x)=\int_{-\infty}^{+\infty}\frac{sin((t-s)\sqrt{-\Delta}) }{\sqrt{-\Delta}}F(s,x)ds
        \end{equation}
	\end{proof}	
	  Let $T_1(t)=\frac{sin(t\sqrt{-\Delta})}{\sqrt{-\Delta}}$. In fact, $W=\sqrt{-\Delta}T_1T_1^*$. By (\ref{eq:Stric01}) and (\ref{eq:Stric02}), we easily check that W is bounded from $L^{p'}_tL^{q'}_x$ to $L^p_tL^q_x$. Christ-Kiselev lemma\cite{christ2001maximal} gives us the desired bound for the last piece in (\ref{eq:duh}).
	\section{Energy Increment Estimate} 
	
	In this section we apply the fourier truncation method to (\ref{eq:01}) and derive an energy estimate for the low mode. Roughly speaking, we divide the time interval I into small subintervals $I_j$, the size of which is measured by the scattering norm $\|u\|_{L^4_{t,x}}^4$. In other words, $\|u\|^4_{L^4_{t,x}}(I_j)<\epsilon $ for some constant $\epsilon>0$ small enough. We split the initial data into high and low frequency parts. High frequency data will be evolved linearly globally. On the first step $I_1$, low frequency data will be evolved by the cubic nonlinear evolution. The original solution on $I_1$ will be equal to the sum of the high mode, low mode and a correction term, which starts with zero initial data. The low mode enjoys energy conservation on $I_1$. Moreover, we prove that the correction term has finite energy on $I_1$. Thus, we may add the correction term to the low mode at the end of 
	$I_1$ and start the procedure on the next step $I_2$, and so on. Section 3.1 is devoted to local energy estimates of the correction term on small time intervals. Section 3.2 treat the solution on larger intervals by summing up the subintervals. 	\\

	\subsection{Local Estimates on Small Interval}\hfill \\
	In this subsection, we study the solution to (\ref{eq:01}) with initial data $u_0$ on a small interval $I_1$, where the scattering norm of the solution is bounded by 
	\begin{equation}
		\|u\|_{L^4_{t,x}}^4< \epsilon
	\end{equation}$0<\epsilon \ll 1$.\\
	To begin with, fix $s_0>0$ which is to be determined, we decompose the initial data $(u_0,u_1)$ by
	\begin{equation}
		\begin{aligned}
		   &(u_0,u_1)=(u_{0,hi},u_{1,hi})+(u_{0,lo},u_{1,lo})	\\
		   &u_{j,hi}=P_{\leq s_0}u_j\\
		   &u_{j,lo}=P_{>s_0}u_j
		\end{aligned}
	\end{equation}
	
	By \ref{BernIneq}, we have that
	\begin{equation}\label{eq:03}
	  \begin{aligned}
	  	  &\|(u_{0,hi},u_{1,hi})\|_{H^\nu\times H^{\nu-1}}\lesssim s_0^{\frac{1}{2}(s-\nu)} 	\|(u_0,u_1)\|_{H^s\times H^{s-1}}\\
	  	  &\|(u_{0,lo},u_{1,lo})\|_{H^1\times L^2}\lesssim s_0^{-\frac{1}{2}(1-s)} 	\|(u_0,u_1)\|_{H^s\times H^{s-1}}
	  \end{aligned}
	\end{equation} for $\nu\leq s$

	We evolve the high and low frequency data respectively by the following equation:
	
	\begin{equation}
	    \left\{\begin{aligned}
	         &\psi_{tt}-\Delta_{\mathbb{H}^3}\psi=0\\
	         &(\psi(0),\psi_t(0))=(u_{0,hi},u_{1,hi})\\
        \end{aligned}\right.
	\end{equation}

	\begin{equation}
	    \left\{\begin{aligned}
	         &\phi_{tt}-\Delta_{\mathbb{H}^3}\phi+|\phi|^3=0\\
	         &(\phi(0),\phi_t(0))=(u_{0,lo},u_{1,lo})\\
         \end{aligned}\right.
	\end{equation}	

    The correction term, denoted v, satisfies the following equation:
	\begin{equation}\label{eq:05}
	    \left\{\begin{aligned}
	         &v_{tt}-\Delta_{\mathbb{H}^3}v+|v|^3=-|u|^3+|\phi|^3+|v|^3=\mathcal{O}(\psi^3+\psi\phi^2+\psi v^2+\phi^2 v+\phi v^2)\\
	         &(v(0),v_t(0))=(0,0)\\
         \end{aligned}\right.
	\end{equation}

	By Strichartz estimates, we have 
	\begin{equation}\label{eq:04}
	   \|\psi\|_{L^4_{t,x}}\lesssim s_0^{\frac{1}{2}(s-\frac{1}{2})}	
	\end{equation}
    $\phi$ has conserved energy on $I_1$. In fact,
    \ref{eq:03} and Sobolev embedding give the bound
	\begin{equation}\label{energycondition}
		E(\phi)\lesssim s_0^{1-s}
	\end{equation}
	
	In addition, we have the following estimates on some of the space-time norms.
	\begin{prop}\label{prop:01}
	   \hfill\begin{enumerate}
	   	\item $\|v\|_{L^4_{t,x}(I_1)}\lesssim s_0^{\frac{1}{2}(s-\frac{1}{2})}$
	   	\item $\|\phi\|_{L^4_{t,x}(I_1)}\lesssim \epsilon$
	   	\item $\|v\|_{L^\frac{8}{3}_tL^8_x(I_1)}\lesssim s_0^{\frac{1}{2}(s-\frac{3}{4})}$ if $s\geq \frac{3}{4}$
	   	\item $\|\phi\|_{L^\frac{8}{3}_tL^8_x(I_1)}\lesssim s_0^{-\frac{1}{2}(1-s)}$
	   \end{enumerate}
	\end{prop}
    \begin{proof}
       	Plug (\ref{eq:04}) into $u=\psi+\phi+v$, we have
       	\begin{equation}\label{eq:06}
       		\|\phi\|_{L^4_{t,x}}\leq \epsilon+\|v\|_{L^4_{t,x}}
       	\end{equation}
       	By Strichartz estimates and \ref{eq:05}, we have
       	\begin{equation}\label{eq:07}
       	  \begin{aligned}
       	     \|v\|_{L^4_{t,x}}&\lesssim \|v^3+\psi^3+\psi\phi^2+\psi v^2+\phi^2 v+\phi v^2\|	_{L^\frac{4}{3}_{t,x}}\\
       	     &\lesssim \|v\|^3_{L^4_{t,x}}+\|\psi\|^3_{L^4_{t,x}}+\|\psi\|_{L^4_{t,x}}\|\phi\|^2_{L^4_{t,x}}+\|\psi\|_{L^4_{t,x}}\|v\|^2_{L^4_{t,x}}+\|\phi\|^2_{L^4_{t,x}}\|v\|_{L^4_{t,x}}+\|\phi\|_{L^4_{t,x}}\|v\|^2_{L^4_{t,x}}\\
       	  \end{aligned}
       	\end{equation}
       	Plug (\ref{eq:04}),(\ref{eq:06}) into (\ref{eq:07}),
       	\begin{equation}
       	   \|v\|_{L^4_{t,x}}\lesssim s_0^{\frac{1}{2}(s-\frac{1}{2})}+o(1)\|v\|_{L^4_{t,x}}+	High\ Order\ Terms\ of\ \|v\|_{L^4_{t,x}}
       	\end{equation}
       	Thus we get our estimate on $\|v\|_{L^4_{t,x}}$ as desired.
       	The estimate on $\|\phi\|_{L^4_{t,x}}$ follows from \ref{eq:06}
       	To get (3), we use Strichartz estimate again similar to what we did in \ref{eq:07}.
       	To get (4), note that 
       	\begin{equation}
       		\|\phi(0),\phi_t(0)\|_{H^\frac{3}{4}\times H^{-\frac{1}{4}}}\lesssim \|\phi(0),\phi_t(0)\|_{H^1\times L^2}\lesssim s_0^{-\frac{1}{2}(1-s)}
       	\end{equation}
        then (4) follows from the same Strichartz estimate argument.
    \end{proof}
    
    In order to ensure the solution can be extended beyond $I_1$, we need to make sure the energy of $\phi$ and v stay bounded.
	
	As for v, even though the forcing term of (\ref{eq:05}) contains rough terms, e.g. $\psi^3$, v turns out to lie in $L^\infty_tH^1_x$ throughout $I_1$. This smoothing property is crucial for the fourier truncation method.
	In fact, 
	\begin{prop}\label{prop2}
	The energy of v on $I_1$ is bounded by
	\begin{equation}
		\sup_{t\in I_1}E(v)(t) \lesssim s_0^{\frac{7}{4}s-\frac{3}{2}
		}
	\end{equation}	
	\end{prop}
    \begin{proof}
       Take time derivative of E(v), we get
       \begin{equation}
          \begin{aligned}
       	     \frac{dE(v)}{dt}&=Re<v_{tt}-\Delta v+|v|^3,v_t> \\
       	     &\lesssim <\mathcal{O}(\psi^3+\psi\phi^2+\psi v^2+\phi^2 v+\phi v^2),v_t>         \\
       	     &\lesssim(\|\psi\|_6^3+\|v\|_8\|v\|_4\|\psi\|_8+\|v\|_8^2\|\phi\|_4+\|\psi\|_8\|\phi\|_8\|\phi\|_4+\|\phi\|_8\|v\|_8\|\phi\|_4)E(v)^\frac{1}{2}\\	       	     
          \end{aligned}
       \end{equation}	
       Integrating in time and use H\"{o}lder inequality, we have
       \begin{equation} \label{eq:08}
          \begin{aligned}
              \sup_{t\in I_1}E(v)^\frac{1}{2}&\lesssim 
              \|\psi\|_{L^3_tL^6_x}^3+
              \|v\|_{L^\frac{8}{3}_tL^8_x}^2\|\psi\|_{L^4_{t,x}}+
              \|v\|_{L^\frac{8}{3}_tL^8_x}^2\|\phi\|_{L^4_{t,x}}\\
              &+\|\psi\|_{L^\frac{8}{3}_tL^8_x}\|\phi\|_{L^\frac{8}{3}_tL^8_x}\|\phi\|_{L^4_{t,x}}+
              \|\phi\|_{L^\frac{8}{3}_tL^8_x}\|v\|_{L^\frac{8}{3}_tL^8_x}\|\phi\| _{L^4_{t,x}}        	
          \end{aligned}
       \end{equation}
       Plug the results of Prop \ref{prop:01} into (\ref{eq:08}),
       we get the desired estimate for E(v).

    \end{proof}
    
    \subsection{Conditional Energy Increment Estimate}
	
	Suppose the solution exists on time interval I with $\|u\|^4_{L^4_{t,x}(I)}\leq M$ for some $M>0$.
	Divide I into subintervals $I_j$ such that on each $I_j$ $\|u\|_{L^4_{t,x}(I_j)}^4\leq\epsilon$ as in section 3.1. $\phi$ and v are defined on $I_1$ as in section 3.1. We shall extend $\phi$ and v to I inductively. Suppose $\phi$ and v are defined on $I_{j-1}=(b_{j-1},b_j)$. Write $f(b_j^+)=\lim_{t\searrow b_j}f(t)$ and $f(b_j^-)=\lim_{t\nearrow b_j}f(t)$, we define $\phi$ and v on the next step $I_j=(b_j,b_{j+1})$ as following	\begin{equation*}
		\left\{\begin{aligned}
			&\phi_{tt}+\Delta_{\mathbb{H}^3}\phi+|\phi|^3=0\\
			&\phi(b_j^+,x)=\phi(b_j^-,x)+v(b_j^-,x)\\
			&\phi_t(b_j^+,x)=\phi_t(b_j^-,x)+v_t(b_j^-,x)
		\end{aligned}
		\right.
	\end{equation*}
	and
	\begin{equation*}
		\left\{\begin{aligned}
			&v_{tt}+\Delta_{\mathbb{H}^3}v+|u|^3-|\phi|^3=0\\
			&v(b_j^+,x)=0\\
			&v_t(b_j^+,x)=0
		\end{aligned}
		\right.
	\end{equation*}

    To estimate the energy gained throughout the whole interval I, we will estimate it locally on each $I_j$ and sum them up. Obviously, we wish that our results Prop \ref{prop:01} and Prop \ref{prop2} in section 3.1 remain true for not only $I_1$ but all the subsequent subintervals $I_j$. Since $\phi$ and v on each $I_j$ are defined by the same equations, it suffices to have the condition (\ref{energycondition}) remain valid throughout I. This conditional energy estimate is stated in the following proposition,
    
	\begin{prop}
		Suppose the solution u exists on interval I  such that $\|u\|^4_{L^4_{t,x}(I)}\leq M$ where $M\sim s_0^{-\frac{3}{16}s+\frac{1}{8}}$. Then the estimates in Prop \ref{prop:01} hold for all $I_j$ and we have 
		\begin{equation}
			\sup_{t\in I}E(\phi(t))\leq E(\phi(0))+C\frac{M}{\epsilon}s_0^{\frac{19}{16}s-\frac{9}{8}}\sim s_0^{-(1-s)}
		\end{equation}
	\end{prop}
	Proof:
	Suppose the condition (\ref{energycondition}) holds on I.
	We divide I into $I_j$'s as we have stated. The number of such subintervals is at most  $\mathcal{O}(\frac{M}{\epsilon})$. The energy increment of $\phi$ through each $I_j$ is 
	\begin{equation}
		\begin{aligned}
             \Delta E&=E(\phi+v)-E(\phi)\\
             &\lesssim \int_{\mathbb{H}^3} (|\nabla\phi \cdot \nabla v|+|\phi_t v_t|+ |\nabla v|^2+|v_t|^2 d\mu+ |\phi^3v|+|v|^4)d\mu\\
             &\leq E(\phi)^\frac{1}{2}E(v)^\frac{1}{2}+E(\phi)^\frac{3}{4}E(v)^\frac{1}{4}+E(v)\\
             &\lesssim s_0^{-\frac{3}{4}(1-s)}\cdot s_0^{\frac{1}{4}(\frac{7}{4}s-\frac{3}{2})}\\
             &\lesssim s_0^{\frac{19}{16}s-\frac{9}{8}}
		\end{aligned}
	\end{equation}
	
	Thus we obtain the energy increment of $\phi$ on each subinterval $I_j$. To justify our assumption (\ref{energycondition}) with a bootstrap argument, we need only the following inequality.
	\begin{equation*}
		\frac{M}{\epsilon}\sup_j\Delta E(\phi)\lesssim\frac{M}{\epsilon}s_0^{\frac{19}{16}s-\frac{9}{8}}\lesssim s_0^{-(1-s)}
	\end{equation*}
	\begin{equation}
		\implies M\sim s_0^{-\frac{3}{16}s+\frac{1}{8}}
	\end{equation}
	Thus arises our assumption on the size of M.

	\section{Global Wellposedness and Scattering}
	We will show that there exists $s_0$ small, such  that 
	\begin{equation}
		\|u\|^4_{L^4_{t,x}(I)}\leq M 
	\end{equation}
	for any $I=[0,T]$ where the solution exists. The global well-posedness and scattering follow from this bound. In order to bound the $L^4_{t,x}$ norm of $\zeta$, we need the following Morawetz estimate, which adds the influence of the correction terms to the one in \cite{shen2016semi}.
	
		\begin{prop}[Morawetz estimate for modified equation]
		Suppose on time interval I, u solves the following modified cubic NLW with error term $\mathcal{N}$,
		\begin{equation*}
			u_{tt}-\Delta_{\mathbb{H}^3}u+u^3=\mathcal{N}
		\end{equation*}
		then $\zeta$ satisfies the following Morawetz estimate
		\begin{equation}\label{morawetz}
			\|\zeta\|_{L^{4}_{t,x}}^{4}\lesssim \sup_{t\in I}E(\zeta)+\|\mathcal{N}\zeta\|_{L^1_{t,x}(I)}+\|\mathcal{N}\nabla \zeta\|_{L^1_{t,x}(I)}
	    \end{equation}
	\end{prop}
    \begin{proof}
        For simplicity, we adopt an informal proof, assuming all integrals in our proof are convergent. Let the Morawetz potential M(t) be defined as
        \begin{equation}\label{eq:morawetz}
            M(t)=-\int_{\mathbb{H}^3}u_tD^\alpha a D_\alpha u+u_tu\cdot \frac{\Delta a}{2}d\mu
        \end{equation}
        where the a is the same function in \cite{ionescu2009semilinear}. a satisfies 
        \begin{equation}
            \left\{\begin{aligned}
                &\Delta a=1\\
                &\|\nabla a\|\leq C\\
                &D^2a\ is\ positive\ definite
            \end{aligned}\right.
        \end{equation}
        Differentiate \ref{eq:morawetz}, we get
        \begin{equation}
            \begin{aligned}
                \frac{d}{dt}M(t)&=-\int u_{tt}D^\alpha aD_\alpha u+u_tD^\alpha aD_\alpha u_t+u_{tt}u\frac{\Delta a}{2}+u_t^2\frac{\Delta a}{2}d\mu\\
                &=I+II+III+IV
            \end{aligned}
        \end{equation}
        \begin{equation*}
            \begin{aligned}
                 I&=-\int u_{tt}D^\alpha aD_\alpha u\\
                  &=-\int (\Delta u-u^3+\mathcal{N})D^\alpha aD_\alpha u\\
                  &=-\int D^\beta D_\beta u D^\alpha aD_\alpha u-\frac{1}{4}D^\alpha a D_\alpha(u^4)+\mathcal{N}D^\alpha a D_\alpha u\\
                  &=\int D^\beta D^\alpha a D_\alpha u D_\beta u+\frac{1}{2}D^\alpha aD_\alpha(|\nabla u|^2)-\frac{1}{4}D^\alpha a D_\alpha(u^4)+\mathcal{N}D^\alpha a D_\alpha u\\
                  &=\int D^\beta D^\alpha a D_\alpha u D_\beta u - \frac{1}{2}\Delta a|\nabla u|^2-\frac{1}{4}\Delta a u^4-\mathcal{N}D^\alpha a D_\alpha u\\
                  &\geq -\frac{1}{2}\|\nabla u\|^2_2-\frac{1}{4}\|u\|^4_4-\|\mathcal{N}\nabla u\|_1\\
                II&=-\int u_tD^\alpha aD_\alpha u_t\\
                &=-\int \frac{1}{2}D^\alpha a D_\alpha u_t^2\\
                &=\int \frac{1}{2}\Delta au_t^2\\
                &=\frac{1}{2}\|u_t\|^2_2\\
                III&=-u_{tt}u\frac{\Delta a}{2}\\
                &=-\int \frac{1}{2}(\Delta u-u^3+\mathcal{N})u\\
                &=\frac{1}{2}\|\nabla u\|^2_2+\frac{1}{2}\|u\|^4_4-\frac{1}{2}\int \mathcal{N}u\\
                &\geq \frac{1}{2}\|\nabla u\|^2_2+\frac{1}{2}\|u\|^4_4-\|\mathcal{N}u\|_1\\
                IV&=-\frac{1}{2}\|u_t\|^2_2                
            \end{aligned}
        \end{equation*}
        Summing I-IV, we get
        \begin{equation}
            \frac{d}{dt}M(t)\geq \frac{1}{4}\|u\|_4-\|\mathcal{N}u\|_1-\|\mathcal{N}\nabla u\|_1
        \end{equation}
        Integrating in time, we get
        \begin{equation}
            \begin{aligned}
                \|u\|^4_{L^4_{t.x}}&\leq \sup_{t\in I}|M(t)|+\|\mathcal{N}u\|_{L^1_{t,x}}+\|\mathcal{N}\nabla u\|_{L^1_{t,x}}\\
                &\lesssim \sup_{t\in I}E(u)(t)+\|\mathcal{N}u\|_{L^1_{t,x}}+\|\mathcal{N}\nabla u\|_{L^1_{t,x}}
            \end{aligned}
        \end{equation}
        
    \end{proof}
    
    To prove the boundedness of $\|u\|_{L^4_{t,x}}$, we will use a bootstrap argument. In fact, we will prove the following implication:

	\begin{prop}\label{prop3}
		For $s>\frac{166}{185}$, there exists $s_0>0$ small enough and $M\sim s_0^{-\frac{3}{16}s+\frac{1}{8}}$, such that the following implication holds:
		\begin{equation*}
			\|u\|^4_{L^4_{t,x}(I)}\leq M \implies \|u\|^4_{L^4_{t,x}(I)}\leq \frac{1}{2}M
		\end{equation*}
	\end{prop}
	
	Denote $\zeta=\phi+v$. On [0,T], $\zeta$ satisfies the modified NLW equation
	\begin{equation}
		\zeta_{tt}+\Delta_{\mathbb{H}^3}\zeta+|\zeta|^3=\mathcal{N}
	\end{equation}
	where $\mathcal{N}=|u|^3-|\zeta|^3=\mathcal{O}(\psi^3+\psi\zeta^2)$.\\ \\
	Due to our choice of M, the result of Proposition 3.3 applies. We have 
	\begin{equation}
		\sup_{t\in I}E[\zeta(t)]\lesssim s_0^{-(1-s)}
	\end{equation}

	\begin{proof}
      Due to our choice of M, Prop (\ref{prop2}) is applicable. It remains to bound the other two terms on the RHS of (\ref{morawetz}).
      
      \begin{equation}
      	\begin{aligned}
      	   \|\mathcal{N}\zeta\|_{L^1_{t,x}(I)}
      	   &\lesssim \|\psi^3\zeta\|	_{L^1_{t,x}(I)}+
      	   \|\psi\zeta^3\|\\
      	   &\lesssim \delta \|\zeta\|_{L^4_{t,x}(I)}^4+C_\delta \|\psi\|^4_{L^4_{t,x}(I)}
      	\end{aligned}
      \end{equation}
      We used Young's inequality in the last inequality. For $\delta$ small enough, the $\delta \|\zeta\|_{L^4_{t,x}(I)}^4$ will be absorbed into the LHS of \ref{morawetz}. The $\|\psi\|_{L^4_{t,x}(I)}^4$ is bounded by $s_0^{2(s-\frac{1}{2})}$.\\
      
      The other term is estimated as follows:
      \begin{equation}
      	\begin{aligned}
      	   \|\mathcal{N}\nabla \zeta\|_{L^1_{t,x}(I)}
      	   &\lesssim (\|\psi\|_{L^3_tL^6_x(I)}^3+\|\psi\|_{L^\frac{8}{3}_tL^8_x(I)}\|\zeta\|_{L^\frac{8}{3}_tL^8_x(I)}\|\zeta\|_{L^4_{t,x}(I)})\|\nabla\zeta\|_{L^\infty_tL^2_x(I)}\\
      	\end{aligned}
      \end{equation}
      We have that $\|\psi\|_{L^3_tL^6_x(I)}\lesssim s_0^{\frac{1}{2}(s-\frac{2}{3})}$ and $\|\psi\|_{L^\frac{8}{3}_tL^8_x(I)}\lesssim s_0^{\frac{1}{2}(s-\frac{3}{4})}$ by Strichartz estimate. $\|\zeta\|_{L^4_{t,x}(I)}\leq M^\frac{1}{4}$ by assumption. To estimate $\|\zeta\|_{L^\frac{8}{3}_tL^8_x(I)}$, we apply Prop \ref{prop:01} on each subinterval and sum them up. There are $\leq \frac{M}{\epsilon}$ subintervals, thus
      \begin{equation}
        \begin{aligned}
            \|\zeta\|_{L^\frac{8}{3}_tL^8_x(I)}&\lesssim [\frac{M}{\epsilon}(s_0^{-\frac{1}{2}(1-s)})^\frac{8}{3}]^\frac{3}{8}\\
            &\lesssim M^\frac{3}{8}s_0^{-\frac{1}{2}(1-s)}
        \end{aligned}
      \end{equation}
      Collecting these terms, we have
      \begin{equation}
           \|\mathcal{N}\nabla \zeta\|_{L^1_{t,x}(I)}\lesssim s_0^{\frac{3}{2}s-\frac{11}{8}}M^\frac{5}{8}
      \end{equation}

	\end{proof}

	Plugging these into the Morawetz inequality, 
	\begin{equation}
		\|\zeta\|_{L^4_{t,x}}^4\lesssim s_0^{1-s}+\delta\|\zeta\|_{L^4_{t,x}}^4+C_\delta s_0^{2s-\frac{1}{2}}+s_0^{\frac{3}{2}s-\frac{11}{8}}M^\frac{5}{8}
	\end{equation}
	
	After absorbing $\epsilon\|\zeta\|_{L^4_{t,x}}^4$ into left hand side, we get 
	\begin{equation}\label{eq:final}
		\|\zeta\|_{L^4_{t,x}}^4\lesssim s_0^{\frac{3}{2}s-\frac{11}{8}}M^\frac{5}{8}
	\end{equation}
	
	In order for Prop \ref{prop3} to hold, we require the right hand side of \ref{eq:final} to be bounded by $\frac{1}{2}M\sim s_0^{-\frac{3}{16}s+\frac{1}{8}}$, i.e.
	\begin{equation}
		s_0^{\frac{3}{2}s-\frac{11}{8}}\cdot s_0^{\frac{5}{8}(-\frac{3}{16}s+\frac{1}{8})}\lesssim s_0^{(-\frac{3}{16}s+\frac{1}{8})}
	\end{equation}
	This can be achieved by $s_0$ small enough as long as the exponents satisfy \begin{equation}
	    \frac{3}{2}s-\frac{11}{8}+\frac{5}{8}(-\frac{3}{16}s+\frac{1}{8})>-\frac{3}{16}+\frac{1}{8}
	\end{equation}
	which gives us 
    \begin{equation}
        s>\frac{182}{201}\approx 0.905
    \end{equation}

    \begin{cor}[Global Wellposedness and Scattering]
        When $s>\frac{182}{201}$, The initial value problem (\ref{eq:01}) are globally wellposed and the solutions scatter in $L^4_{t,x}\cap C_tH_x^\frac{1}{2}$. The scattering norm $\|u\|_{L^4_{t,x}}\leq f(s,\|u_0\|_{H^s},\|u_1\|_{H^{s-1}})$ for some function $f>0$.
    \end{cor}
    \begin{proof}
        The bound for scattering norm follows from Prop \ref{prop3} and a standard continuity argument. Global wellposedness and scattering are straightforward result of that.
    \end{proof}
	\bibliographystyle{alpha}
	\bibliography{NLSH3.bib}

	
\end{document}